\documentclass[reqno,10pt]{article}

\usepackage{a4wide}
\usepackage{color}
\usepackage{mathrsfs}
\usepackage{mathtools}
\usepackage{amsmath}
\usepackage{amsthm}
\usepackage{amssymb}
\usepackage{bbm}
\usepackage{tikz-cd}
\usepackage{esint}
\usepackage{nicefrac}
\numberwithin{equation}{section}
\usepackage[colorlinks,citecolor=blue,linkcolor=red]{hyperref}
\usepackage{longtable}
\usepackage{setspace}

\usepackage[latin1]{inputenc}

\newcommand{\N}{\mathbb{N}}
\newcommand{\R}{\mathbb{R}}
\newcommand{\sfd}{{\sf d}}
\renewcommand{\d}{{\mathrm d}}
\newcommand{\X}{{\rm X}}
\newcommand{\Y}{{\rm Y}}
\newcommand{\mm}{\mathfrak{m}}
\newcommand{\1}{\mathbbm 1}
\newcommand{\LIP}{{\rm LIP}}
\newcommand{\Lip}{{\rm Lip}}
\newcommand{\lip}{{\rm lip}}
\renewcommand{\div}{{\rm div}}

\newcommand{\Per}{{\rm Per}}
\newcommand{\dist}{{\rm dist}}
\newcommand{\diam}{{\rm diam}}

\newcommand{\fr}{\penalty-20\null\hfill\(\blacksquare\)}

\newtheorem{theorem}{Theorem}[section]

\newtheorem{lemma}[theorem]{Lemma}
\newtheorem{proposition}[theorem]{Proposition}
\newtheorem{definition}[theorem]{Definition}

\newtheorem{remark}[theorem]{Remark}

\title{A note on Laplacian bounds, deformation properties \\ and isoperimetric sets in metric measure spaces}
\author{Enrico Pasqualetto\footnote{\href{mailto:enrico.e.pasqualetto@jyu.fi}{enrico.e.pasqualetto@jyu.fi},
Department of Mathematics and Statistics, P.O.\ Box 35 (MaD), FI-40014 University of Jyv\"askyl\"a, Finland.}
\;and Tapio Rajala\footnote{\href{mailto:tapio.m.rajala@jyu.fi}{tapio.m.rajala@jyu.fi},
Department of Mathematics and Statistics, P.O.\ Box 35 (MaD), FI-40014 University of Jyv\"askyl\"a, Finland.}}
\begin{document}

\date{\today}
\maketitle

\begin{abstract}
In the setting of length PI spaces satisfying a suitable deformation property, it is known that each isoperimetric set
has an open representative. In this paper, we construct an example of a length PI space (without the deformation
property) where an isoperimetric set does not have any representative whose topological interior is non-empty.
Moreover, we provide a sufficient condition for the validity of the deformation property, consisting in an upper
Laplacian bound for the squared distance functions from a point. Our result applies to essentially non-branching
${\sf MCP}(K,N)$ spaces, thus in particular to essentially non-branching ${\sf CD}(K,N)$ spaces and to many
Carnot groups and sub-Riemannian manifolds. As a consequence, every isoperimetric set in an essentially non-branching
${\sf MCP}(K,N)$ space has an open representative, which is also bounded whenever a uniform lower bound on the volumes
of unit balls is assumed.
\end{abstract}

\section{Introduction}
\subsection{General overview}
Functions of bounded variation and sets of finite perimeter have been studied extensively even in the non-smooth
context of metric measure spaces, starting from \cite{Mir:03}. Whereas a solid and consistent theory of BV functions
is available on arbitrary (possibly infinite-dimensional) metric measure spaces \cite{Amb:DiMa:14}, finer results
hold on those metric measure spaces that are uniformly locally doubling and support a weak local \((1,1)\)-Poincar\'{e}
inequality, which we call \emph{PI spaces}; see
e.g.\ \cite{Amb:01,Amb:02,Kin:Kor:Lor:Sha:13,Kin:Kor:Sha:Tuo:14,Lah:20}.
A distinguished class of sets of finite perimeter is the one of \emph{isoperimetric sets}, i.e.\ of minimisers
of the perimeter under a volume constraint. In the series of papers 
\cite{Ant:Fog:Poz:24,Ant:Pas:Poz:22,Ant:Bru:Fog:Poz:22,Ant:Nar:Poz:22,Ant:Pas:Poz:Sem:22,Ant:Pas:Poz:Sem:24}
(see also the survey \cite{Poz:22}), the isoperimetric problem has been investigated in the framework of
\emph{\({\sf RCD}(K,N)\) spaces}, which are finite-dimensional infinitesimally Hilbertian metric measure spaces
satisfying lower synthetic Ricci curvature bounds in the sense of Lott--Sturm--Villani (see \cite{Amb:18} and
the references therein). Topological properties of isoperimetric sets in non-collapsed \({\sf RCD}(K,N)\) spaces
were a fundamental ingredient for further investigations of the isoperimetric problem.
\medskip

In \cite{Ant:Pas:Poz:Vio:23}, the first named author, together with G.\ Antonelli, M.\ Pozzetta and I.\ Y.\ Violo, studied
isoperimetric sets in the setting of length PI spaces satisfying the so-called \emph{deformation property}. The latter
is a condition that prescribes a control on the increment of perimeters under perturbations with balls, see
\cite[Definition 3.3]{Ant:Pas:Poz:Vio:23}. The main result of \cite{Ant:Pas:Poz:Vio:23} states that, under the above
assumptions, the essential interior of every isoperimetric set is topologically open. Several other properties, such
as boundary density estimates, were obtained as well. The deformation property was known to hold on Euclidean spaces
\cite{Gon:Mas:Tam:83}, Riemannian manifolds and \({\sf RCD}(K,N)\) spaces \cite[Theorem 1.1]{Ant:Pas:Poz:22};
cf.\ with \cite[Remark 3.4]{Ant:Pas:Poz:Vio:23}. The goal of the present paper is twofold:
\begin{itemize}
\item To show that no topological regularity can be expected without the deformation property.
\item To prove that the deformation property is satisfied in a much larger family of spaces.
\end{itemize}
\subsection{Statement of results}
Let \((\X,\sfd,\mm)\) be a metric measure space. We denote by \(\Per(E;\cdot)\) the \emph{perimeter measure}
of a set of finite perimeter \(E\subseteq\X\) and we write \(\Per(E)\coloneqq\Per(E;\X)\); see Definition \ref{def:fin_per}.
Let us consider the \emph{isoperimetric profile} \(I\colon(0,\mm(\X))\to[0,+\infty]\) of the space \((\X,\sfd,\mm)\),
which is given by
\[
I(v)\coloneqq\inf\big\{\Per(E)\;\big|\;E\subseteq\X\text{ of finite perimeter, }\mm(E)=v\big\}
\quad\text{ for every }0<v<\mm(\X).
\]
A set of finite perimeter \(E\subseteq\X\) with \(0<\mm(E)<+\infty\) is said to be an \emph{isoperimetric set}
provided it satisfies \(\Per(E)=I(\mm(E))\); cf.\ with Definition \ref{def:isoper_set}. In the case where
\((\X,\sfd,\mm)\) is a PI space, we know from \cite{Amb:02} that a distinguished \(\mm\)-a.e.\ representative
of a set of finite perimeter \(E\) is its
\[
\text{\emph{essential interior}}\quad E^{(1)}\coloneqq\bigg\{x\in\X\;\bigg|
\;\lim_{r\searrow 0}\frac{\mm(E\cap B(x,r))}{\mm(B(x,r))}=1\bigg\}.
\]
The perimeter measure \(\Per(E;\cdot)\) is concentrated on the \emph{essential boundary}
\(\partial^e E\coloneqq\X\setminus(E^{(1)}\cup E^{(0)})\) of \(E\), where \(E^{(0)}\coloneqq(\X\setminus E)^{(1)}\)
denotes the \emph{essential exterior} of \(E\). It is proved in \cite[Theorem 1.1]{Ant:Pas:Poz:Vio:23} that if
\((\X,\sfd,\mm)\) is a length PI space satisfying the \emph{deformation property} \cite[Definition 3.3]{Ant:Pas:Poz:Vio:23},
then the essential interior \(E^{(1)}\) of every isoperimetric set \(E\subseteq\X\) is topologically open.
Since length PI spaces without the deformation property do exist \cite[Remarks 3.5 and 3.6]{Ant:Pas:Poz:Vio:23},
a natural question arises: are the essential interiors of isoperimetric sets topologically open in every length PI space?
Cf.\ with \cite[Question 1.5]{Ant:Pas:Poz:Vio:23}. Our first main result provides a negative answer to this question.
\begin{theorem}\label{thm:main_1}
There exist a length PI space \((\X,\sfd,\mm)\) and an isoperimetric set \(E\subseteq\X\) such that every
\(\mm\)-a.e.\ representative of \(E\) has empty topological interior.
\end{theorem}

Section \ref{s:main_1} will be devoted to the proof of Theorem \ref{thm:main_1}. We point out that the latter theorem
shows that the topological regularity of isoperimetric sets can fail in a very strong sense when the deformation property
does not hold: it can happen that \emph{no} representative of an isoperimetric set contains not even one interior point.
Furthermore, as we will see in Theorem \ref{thm:example}, an example as in Theorem \ref{thm:main_1} can be required
to satisfy several additional assumptions. In fact, the example we will construct is a \(2\)-dimensional flat torus
\(\R^2\setminus\mathbb Z^2\) equipped with its geodesic distance and a weighted Hausdorff measure
\(\rho\mathcal H^2\), whose density \(\rho\) is bounded away from zero and infinity.
\medskip

Our second main result gives a sufficient condition for the validity of the deformation property, which applies
to a rather vast class of metric measure spaces of interest. Informally, we show that an upper Laplacian bound
on the squared distance function from a point implies a local deformation property. This result is valid on
metric measure spaces \((\X,\sfd,\mm)\) where the metric space \((\X,\sfd)\) is a length space and \emph{proper}
(i.e.\ bounded closed subsets are compact), but no PI assumption is needed. The \emph{Laplacian comparison}
is formulated in a purely variational way: letting
\[
\sfd_x\coloneqq\sfd(x,\cdot)\colon\X\to[0,+\infty)
\]
denote the \emph{distance function} from \(x\in\X\), we write ``\(\Delta\sfd_x^2\leq C\mm\) on \(B(x,R)\)'' for some \(R,C>0\) if
\[
-\int_{B(x,R)}{\rm D}^+f(\nabla\sfd_x^2)\,\d\mm\leq C\int_{B(x,R)}f\,\d\mm\quad
\begin{array}{ll}
\text{ whenever }f\colon\X\to[0,+\infty)\text{ is}\\
\text{ Lipschitz with bounded support,}
\end{array}
\]
where the auxiliary function \({\rm D}^+f(\nabla\sfd_x^2)\colon\X\to\R\) is defined as
\[
{\rm D}^+f(\nabla\sfd_x^2)(y)\coloneqq\inf_{\varepsilon>0}\frac{\lip(\sfd_x^2+\varepsilon f)^2(y)-\lip(\sfd_x^2)^2(y)}{2\varepsilon}
\quad\text{ for every }y\in\X;
\]
see Definitions \ref{def:Lapl_bound} and \ref{def:Df(nabla_g)}. Note that we are not defining an actual Laplacian,
but only some sort of upper Laplacian bound, and that no Hilbertianity assumptions are required in our definition.
With this said, we can now give the precise statement of our second main result.
\begin{theorem}\label{thm:main_2}
Let \((\X,\sfd,\mm)\) be a metric measure space with \((\X,\sfd)\) proper and length. Assume that \(\Delta\sfd_x^2\leq C\mm\) on
\(B(x,R)\) for some \(x\in\X\) and \(R,C>0\). Let \(E\subseteq\X\) be a set of finite perimeter such that \(\mm(E)<+\infty\). Then
\begin{equation}\label{eq:def_prop_1}
\Per(B(x,r);E)\leq C\frac{\mm(E\cap B(x,r))}{2r}+\Per(E;B(x,r))\quad\text{ for a.e.\ }r\in(0,R).
\end{equation}
If in addition \((\X,\sfd,\mm)\) is a PI space, then the following \emph{deformation property} holds:
\begin{equation}\label{eq:def_prop_2}
\Per(E\setminus B(x,r))\leq C\frac{\mm(E\cap B(x,r))}{2r}+\Per(E)\quad\text{ for every }r\in(0,R).
\end{equation}
\end{theorem}

The proof of Theorem \ref{thm:main_2}, which we will give in Section \ref{s:main_2}, relies on
a calibration-type argument. Heuristically, the deformation property follows from the Laplacian bound
on \(\sfd_x^2\) by the \emph{divergence theorem}: assuming that the space under consideration is a
Riemannian manifold, we have that
\[\begin{split}
C\mm(E\cap B(x,r))&\geq\int_{E\cap B(x,r)}\Delta\sfd_x^2\,\d\mm=\int_{E\cap B(x,r)}\div\nabla\sfd_x^2\,\d\mm\\
&=\int\nu_{E\cap B(x,r)}\cdot\nabla\sfd_x^2\,\d\Per(E\cap B(x,r);\cdot)\\
&=2\int_{E^{(1)}}\sfd_x\,\nabla\sfd_x\cdot\nabla\sfd_x\,\d\Per(B(x,r);\cdot)+2\int_{B(x,r)}\sfd_x\,\nu_E\cdot\nabla\sfd_x\,\d\Per(E;\cdot)\\
&\geq 2r\,\Per(B(x,r);E^{(1)})-2r\,\Per(E;B(x,r))\quad\text{ for a.e.\ }r\in(0,R),
\end{split}\]
where \(\nu_{E\cap B(x,r)}\) and \(\nu_E\) denote the outer unit normals to \(E\cap B(x,r)\) and \(E\), respectively.
In \cite{Ant:Pas:Poz:22}, the above strategy has been generalised to the non-smooth setting of \({\sf RCD}(K,N)\) spaces,
relying on the very refined calculus that is available in that context \cite{Gig:18}. Two essential ingredients were
the Laplacian comparison for squared distance functions in \({\sf RCD}(K,N)\) spaces \cite[Theorem 5.14]{Gig:15} and
the \emph{Gauss--Green integration-by-parts formula} obtained in \cite[Theorem 2.4]{Bru:Pas:Sem:23}. The proof of
Theorem \ref{thm:main_2} we will present follows along the same lines as in \cite[Theorem 1.1]{Ant:Pas:Poz:22},
but -- due to various technical simplifications in the arguments -- the former is significantly shorter than the latter
and valid in much greater generality. In particular, no refined differential calculus is required.
\medskip

Finally, by combining Theorem \ref{thm:main_2} with results from \cite{Cav:Mon:20,Ant:Pas:Poz:Vio:23}, we obtain the following.
\begin{theorem}\label{thm:MCP}
Let \((\X,\sfd,\mm)\) be an essentially non-branching \({\sf MCP}(K,N)\) space, for some \(K\in\R\) and \(N\in(1,\infty)\).
Let \(E\subseteq\X\) be an isoperimetric set. Then for any given radius \(R>0\) there exists \(C=C(K,N,R)>0\) such that
\eqref{eq:def_prop_2} holds for every \(x\in\X\). In particular, the following hold:
\begin{itemize}
\item[\(\rm i)\)] The essential interior \(E^{(1)}\) and the essential exterior \(E^{(0)}\) of \(E\) are topologically open sets.
\item[\(\rm ii)\)] The essential boundary \(\partial^e E\) of \(E\) coincides with the topological boundary
\(\partial E^{(1)}\) of \(E^{(1)}\).
\item[\(\rm iii)\)] \textbf{Boundary density estimates.} Given any bounded set \(B\subseteq\X\), there exist two
constants \(\bar r=\bar r(E,B,K,N)>0\) and \(c=c(E,B,K,N)>1\) such that
\[
\frac{1}{c}\leq\frac{\mm(E\cap B(x,r))}{\mm(B(x,r))}\leq 1-\frac{1}{c},
\qquad\frac{1}{c}\leq\frac{r\,\Per(E;B(x,r))}{\mm(B(x,r))}\leq c
\]
hold for every \(x\in\partial^e E\cap B\) and \(r\in(0,\bar r)\).
\item[\(\rm iv)\)] Given any \(K\subseteq\X\) compact, there exist constants \(\Lambda,r_0>0\)
such that \(E\) is a \emph{\((\Lambda,r_0)\)-perimeter minimiser} on \(K\), meaning that
\(\Per(E;B(x,r))\leq\Per(F;B(x,r))+\Lambda\,\mm(E\Delta F)\) holds whenever \(F\subseteq\X\) is
a set of finite perimeter with \(E\Delta F\subseteq B(x,r)\) for some \(x\in K\) and \(r\in(0,r_0)\).
\item[\(\rm v)\)] Given any \(K\subseteq\X\) compact, there exist constants \(L,r_1>0\)
such that \(E\) is \emph{\((L,r_1)\)-quasi minimal} on \(K\), meaning that
\(\Per(E;B(x,r))\leq L\,\Per(F;B(x,r))\) holds whenever \(F\subseteq\X\) is a set of finite perimeter
with \(E\Delta F\subseteq B(x,r)\) for some \(x\in K\) and \(r\in(0,r_1)\).
\item[\(\rm vi)\)] \textbf{Boundedness.} Assuming in addition that \(\inf_{x\in\X}\mm(B(x,1))>0\),
the set \(E^{(1)}\) is bounded.
\end{itemize}
\end{theorem}

We recall that \emph{\({\sf MCP}(K,N)\) spaces} are PI metric measure spaces that satisfy the so-called
\emph{measure contraction property} (introduced in \cite{Oht:07} and \cite{Stu:06} independently), while
the \emph{essential non-branching} condition was introduced by the second named author and K.-T.\ Sturm
in \cite{Raj:Stu:12}. The class of essentially non-branching \({\sf MCP}(K,N)\) spaces includes smooth Finsler
manifolds whose Minkowski norms on the tangent spaces are strongly convex and whose Ricci tensor is bounded
from below \cite{Oht:Stu:14}, essentially non-branching \({\sf CD}(K,N)\) spaces \cite{Stu:06} (thus,
\({\sf RCD}(K,N)\) spaces), and several sub-Riemannian spaces, such as the Heisenberg groups \cite{Jui:09},
ideal Carnot groups \cite{Rif:13}, corank-1 Carnot groups \cite{Riz:16}, generalised H-type Carnot groups
\cite{Bar:Riz:18}, the Grushin plane and the Grushin half-plane \cite{Riz:18}, Sasakian foliations
\cite{Bau:Gro:Kuw:Tha:19}, two-step analytic sub-Riemannian structures and Lipschitz Carnot groups \cite{Bad:Rif:20}.
We will prove Theorem \ref{thm:MCP} in Section \ref{s:main_2}, by combining Laplacian comparison results
for distance functions on \({\sf MCP}(K,N)\) spaces \cite{Cav:Mon:20} with Theorem \ref{thm:main_2} and \cite{Ant:Pas:Poz:Vio:23}.
Theorem \ref{thm:MCP} vi) can be applied for example to \cite[Theorem 5.3]{Nob:Vio:24}, showing that the boundedness
assumption made on \(E\) therein can be removed.
\medskip

We conclude the introduction by pointing out that the isoperimetric problem on \({\sf MCP}(K,N)\) spaces has been
studied in \cite{Cav:Man:22}, where a sharp isoperimetric inequality on essentially non-branching \({\sf MCP}(0,N)\)
spaces with Euclidean volume growth was proved. The topological regularity of isoperimetric sets was
proved on Carnot groups \cite{Leo:Rig:03} and on a class of sub-Riemannian manifolds \cite{Gal:Rit:13}.
\subsection*{Acknowledgements}
The first named author was supported by the Research Council of Finland grant 362898.
\section{Preliminaries}
\subsection{Metric spaces and Lipschitz functions}
Let \((\X,\sfd)\) be a metric space. Given any \(x\in\X\) and \(r>0\),
we define the open ball, the closed ball and the sphere of center \(x\) and radius \(r\) as
\[\begin{split}
B(x,r)&\coloneqq\big\{y\in\X\;\big|\;\sfd(x,y)<r\big\},\\
\bar B(x,r)&\coloneqq\big\{y\in\X\;\big|\;\sfd(x,y)\leq r\big\},\\
S(x,r)&\coloneqq\big\{y\in\X\;\big|\;\sfd(x,y)=r\big\},
\end{split}\]
respectively. Note that \(\bar B(x,r)\) contains the topological closure of \(B(x,r)\), while
\(S(x,r)\) contains the topological boundary of \(B(x,r)\).
\medskip

We denote by \(\LIP_{loc}(\X)\) the set of all real-valued locally Lipschitz functions on \(\X\), i.e.\ of those functions
\(f\colon\X\to\R\) such that for any \(x\in\X\) there exists \(r_x>0\) for which \(f|_{B(x,r_x)}\) is Lipschitz.
We then denote by \(\LIP(\X)\) the set of all those functions \(f\in\LIP_{loc}(\X)\) that are globally Lipschitz, while
\(\LIP_{bs}(\X)\) stands for the set of all \(f\in\LIP(\X)\) whose support \({\rm spt}(f)\coloneqq{\rm cl}(\{f\neq 0\})\) is bounded.
Given any \(f\in\LIP_{loc}(\X)\), we define its \textbf{slope function} \(\lip(f)\colon\X\to[0,+\infty)\) as
\[
\lip(f)(x)\coloneqq\limsup_{y\to x}\frac{|f(x)-f(y)|}{\sfd(x,y)}\quad\text{ for every accumulation point }x\in\X
\]
and \(\lip(f)(x)\coloneqq 0\) for every isolated point \(x\in\X\). It can be readily checked that \(\lip(f)\) is Borel.
\medskip

Fix any \(f,g\in\LIP_{loc}(\X)\), \(\lambda\in\R\) and \(\phi\in\LIP(\R)\). One can easily verify that the functions
\(\lambda f\), \(f+g\), \(fg\) and \(\phi\circ f\) belong to \(\LIP_{loc}(\X)\), and that
\begin{subequations}\begin{align}
\label{eq:calc_rul_lip_1}
\lip(\lambda f)&=|\lambda|\lip(f),\\
\label{eq:calc_rul_lip_2}
\lip(f+g)&\leq\lip(f)+\lip(g),\\
\label{eq:calc_rul_lip_3}
\lip(fg)&\leq|f|\lip(g)+|g|\lip(f),\\
\label{eq:calc_rul_lip_4}
\lip(\phi\circ f)&\leq(\lip(\phi)\circ f)\lip(f).
\end{align}\end{subequations}
If in addition we assume that \(\phi\in C^1(\R)\), then we also have that
\begin{equation}\label{eq:chain_lip}
\lip(\phi\circ f)=(|\phi'|\circ f)\lip(f).
\end{equation}
Indeed, the inequality \(\lip(\phi\circ f)\leq(|\phi'|\circ f)\lip(f)\) follows from \eqref{eq:calc_rul_lip_4}, as \(|\phi'|=\lip(\phi)\).
If \(x\in\X\) is such that \(|\phi'|(f(x))=0\), then \(\lip(\phi\circ f)(x)=(|\phi'|\circ f)(x)\lip(f)(x)=0\). If \(|\phi'|(f(x))\neq 0\),
then we can find \(\varepsilon>0\) and \(\psi\in C^1(\R)\) such that \((\psi\circ\phi)(t)=t\) for every \(t\in(f(x)-\varepsilon,f(x)+\varepsilon)\), so that
\[
\lip(f)(x)=\lip(\psi\circ(\phi\circ f))(x)\overset{\eqref{eq:calc_rul_lip_4}}\leq\lip(\psi)((\phi\circ f)(x))\lip(\phi\circ f)(x)
=\frac{\lip(\phi\circ f)(x)}{|\phi'|(f(x))}.
\]
All in all, the identity in \eqref{eq:chain_lip} is proved.
\medskip

We now give the definition of an object that has a key role in our notion of Laplacian bound.
\begin{definition}\label{def:Df(nabla_g)}
Let \((\X,\sfd)\) be a metric space. Let \(f,g\in\LIP_{loc}(\X)\) be given. Then we define
\[\begin{split}
{\rm D}^+f(\nabla g)(x)&\coloneqq\inf_{\varepsilon>0}\frac{\lip(g+\varepsilon f)^2(x)-\lip(g)^2(x)}{2\varepsilon}\quad\text{ for every }x\in\X,\\
{\rm D}^-f(\nabla g)(x)&\coloneqq\sup_{\varepsilon<0}\frac{\lip(g+\varepsilon f)^2(x)-\lip(g)^2(x)}{2\varepsilon}\quad\text{ for every }x\in\X.
\end{split}\]
\end{definition}

The above definition is taken from \cite[Definition 2.1]{Mon:15}, which was in turn inspired by \cite[Definition 3.1]{Gig:15}, where Sobolev functions
and minimal weak upper gradients were used in place of Lipschitz functions and slope functions. In particular, the above Definition \ref{def:Df(nabla_g)}
agrees with \cite[Definition 2.4]{Cav:Mon:20}. In the next result, we collect several basic properties of the functions \({\rm D}^\pm f(\nabla g)\).
\begin{lemma}\label{lem:tech_results_Df(nabla_g)}
Let \((\X,\sfd)\) be a metric space. Let \(f,\tilde f,g\in\LIP_{loc}(\X)\) and \(0\leq\lambda\in\R\) be given. Then:
\begin{itemize}
\item[\(\rm i)\)] Given any \(x\in\X\), the function \(\R\ni\varepsilon\mapsto\lip(g+\varepsilon f)(x)\) is convex.
\item[\(\rm ii)\)] It holds that
\[
-\frac{\lip(g-f)^2}{2}\leq{\rm D}^-f(\nabla g)\leq{\rm D}^+f(\nabla g)\leq\frac{\lip(g+f)^2}{2}.
\]
In particular, the functions \({\rm D}^\pm f(\nabla g)\colon\X\to\R\) are locally bounded.
\item[\(\rm iii)\)] Given any \(x\in\X\), it holds that
\[\begin{split}
{\rm D}^+f(\nabla g)(x)&=\lim_{\varepsilon\searrow 0}\frac{\lip(g+\varepsilon f)^2(x)-\lip(g)^2(x)}{2\varepsilon},\\
{\rm D}^-f(\nabla g)(x)&=\lim_{\varepsilon\nearrow 0}\frac{\lip(g+\varepsilon f)^2(x)-\lip(g)^2(x)}{2\varepsilon}.
\end{split}\]
In particular, the functions \({\rm D}^\pm f(\nabla g)\colon\X\to\R\) are Borel.
\item[\(\rm iv)\)] \({\rm D}^+(\lambda f)(\nabla g)=\lambda\,{\rm D}^+f(\nabla g)\) and \({\rm D}^+(-\lambda f)(\nabla g)=-\lambda\,{\rm D}^-f(\nabla g)\).
\item[\(\rm v)\)] \({\rm D}^\pm f(\nabla f)=\lip(f)^2\), \({\rm D}^+\1_\X(\nabla g)=0\) and \({\rm D}^+(f+\1_\X)(\nabla g)={\rm D}^+f(\nabla g)\).
\item[\(\rm vi)\)] \(|{\rm D}^\pm f(\nabla g)|\leq\lip(f)\lip(g)\).
\item[\(\rm vii)\)] \({\rm D}^+f(\nabla g)={\rm D}^+\tilde f(\nabla g)\) in the interior of \(\{f=\tilde f\}\).
\end{itemize}
\end{lemma}
\begin{proof}
\ \\
{\bf i).} It follows from \eqref{eq:calc_rul_lip_1} and \eqref{eq:calc_rul_lip_2}.\\
{\bf ii), iii).} They are direct consequences of i).\\
{\bf iv), v).} They easily follow from the definition.\\
{\bf vi).} Fix any \(\varepsilon\in\R\) with \(\varepsilon>0\). Applying \eqref{eq:calc_rul_lip_1}
and \eqref{eq:calc_rul_lip_2}, we can estimate
\[\begin{split}
\frac{\lip(g+\varepsilon f)^2-\lip(g)^2}{2\varepsilon}&\leq\frac{(\lip(g)+\varepsilon\,\lip(f))^2-\lip(g)^2}{2\varepsilon}
=\lip(f)\lip(g)+\frac{\varepsilon}{2}\lip(f)^2,\\
\frac{\lip(g-\varepsilon f)^2-\lip(g)^2}{-2\varepsilon}&\geq\frac{(\lip(g)+\varepsilon\,\lip(f))^2-\lip(g)^2}{-2\varepsilon}
=-\lip(f)\lip(g)-\frac{\varepsilon}{2}\lip(f)^2.
\end{split}\]
By letting \(\varepsilon\searrow 0\) in the above, we obtain \({\rm D}^+f(\nabla g)\leq\lip(f)\lip(g)\) and
\({\rm D}^-f(\nabla g)\geq-\lip(f)\lip(g)\), respectively. Taking also ii) into account, we conclude that
\(|{\rm D}^\pm f(\nabla g)|\leq\lip(f)\lip(g)\).\\
{\bf vii).} Just note that \(\lip(g+\varepsilon f)=\lip(g+\varepsilon\tilde f)\) in the interior of \(\{f=\tilde f\}\).
\end{proof} 

In the proof of Theorem \ref{thm:main_2}, we will also need the calculus rules for \({\rm D}^+f(\nabla g)\) that we are going
to state and prove below. We point out that lookalike formulas were obtained in \cite[Propositions 3.15 and 3.17]{Gig:15}.
However, in \cite{Gig:15} the functions \({\rm D}^+f(\nabla g)\) are defined in a different way with respect to our paper.
It also seems that the proof strategies for \cite[Propositions 3.15 and 3.17]{Gig:15} cannot be adapted to our setting, thus
we will provide new arguments in proving Propositions \ref{prop:chain_Dg(nabla_g)} and \ref{prop:Leibniz_Dg(nabla_g)}.
\begin{proposition}\label{prop:chain_Dg(nabla_g)}
Let \((\X,\sfd)\) be a metric space. Let \(f,g\in\LIP_{loc}(\X)\) with \(g\geq 0\) be given. Then
\[
{\rm D}^+f(\nabla g^2)=2g\,{\rm D}^+f(\nabla g).
\]
\end{proposition}
\begin{proof}
Fix any \(x\in\X\). First of all, assume that \(g(x)>0\). Note that for any \(\varepsilon>0\) it holds that
\[
\bigg(g+\frac{\varepsilon(f-f(x))}{2 g(x)}\bigg)^2=(g^2+\varepsilon f)+\underset{\eqqcolon a_{\varepsilon,x}}{\underbrace{\varepsilon\frac{(f-f(x))(g-g(x))}{g(x)}}}
+\underset{\eqqcolon b_{\varepsilon,x}}{\underbrace{\frac{\varepsilon^2(f-f(x))^2}{4 g(x)^2}}}-\varepsilon f(x).
\]
Given that
\[\begin{split}
\lip(a_{\varepsilon,x})(x)&\overset{\eqref{eq:calc_rul_lip_1}}=\frac{\varepsilon}{g(x)}\lip\big((f-f(x))(g-g(x))\big)(x)\\
&\overset{\eqref{eq:calc_rul_lip_3}}\leq\frac{\varepsilon}{g(x)}\big(|f(x)-f(x)|\lip(g)(x)+|g(x)-g(x)|\lip(f)(x)\big)=0,\\
\lip(b_{\varepsilon,x})(x)&\overset{\eqref{eq:calc_rul_lip_1}}=\frac{\varepsilon^2}{4g(x)^2}\lip\big((f-f(x))^2\big)(x)
\overset{\eqref{eq:calc_rul_lip_3}}\leq\frac{\varepsilon^2}{2g(x)^2}|f(x)-f(x)|\lip(f)(x)=0,
\end{split}\]
we deduce that \(\lip(a_{\varepsilon,x})(x)=\lip(b_{\varepsilon,x})(x)=0\), thus accordingly
\[
\lip\bigg(\bigg(g+\frac{\varepsilon(f-f(x))}{2 g(x)}\bigg)^2\bigg)(x)=\lip(g^2+\varepsilon f)(x)\quad\text{ for every }\varepsilon>0.
\]
Therefore, we can compute
\[\begin{split}
{\rm D}^+f(\nabla g^2)(x)&\overset{\phantom{\eqref{eq:chain_lip}}}=\lim_{\varepsilon\searrow 0}\frac{\lip(g^2+\varepsilon f)^2(x)-\lip(g^2)^2(x)}{2\varepsilon}\\
&\overset{\phantom{\eqref{eq:chain_lip}}}=\lim_{\varepsilon\searrow 0}\frac{\lip\big(\big(g+\frac{\varepsilon(f-f(x))}{2 g(x)}\big)^2\big)^2(x)-\lip(g^2)^2(x)}{2\varepsilon}\\
&\overset{\eqref{eq:chain_lip}}=\lim_{\varepsilon\searrow 0}\frac{4g(x)^2\lip\big(g+\frac{\varepsilon(f-f(x))}{2 g(x)}\big)^2(x)-4g(x)^2\lip(g)^2(x)}{2\varepsilon}\\
&\overset{\phantom{\eqref{eq:chain_lip}}}=2g(x)\lim_{\varepsilon\searrow 0}\frac{\lip\big(g+\frac{\varepsilon f}{2 g(x)}\big)^2(x)-\lip(g)^2(x)}{2\frac{\varepsilon}{2g(x)}}
=2g(x)\,{\rm D}^+f(\nabla g)(x),
\end{split}\]
which gives the statement in the case where \(g(x)>0\). Finally, assume that \(g(x)=0\). Then
\[\begin{split}
\big|\lip(g^2+\varepsilon f)(x)-\varepsilon\,\lip(f)(x)\big|&\overset{\eqref{eq:calc_rul_lip_1}}=\big|\lip(g^2+\varepsilon f)(x)-\lip(\varepsilon f)(x)\big|\\
&\overset{\eqref{eq:calc_rul_lip_2}}\leq\lip(g^2)(x)\overset{\eqref{eq:calc_rul_lip_3}}\leq 2g(x)\lip(g)(x)=0,
\end{split}\]
whence it follows that \(\lip(g^2)(x)=0\) and \(\lip(g^2+\varepsilon f)(x)=\varepsilon\,\lip(f)(x)\), thus accordingly
\[
{\rm D}^+f(\nabla g^2)(x)=\lim_{\varepsilon\searrow 0}\frac{\lip(g^2+\varepsilon f)^2(x)-\lip(g^2)^2(x)}{2\varepsilon}=\lim_{\varepsilon\searrow 0}\frac{\varepsilon\,\lip(f)^2(x)}{2}
=0=2g(x)\,{\rm D}^+f(\nabla g)(x).
\]
All in all, the statement is proved.
\end{proof}
\begin{proposition}\label{prop:Leibniz_Dg(nabla_g)}
Let \((\X,\sfd)\) be a metric space. Fix any \(f,g,h\in\LIP_{loc}(\X)\) with \(f,h\geq 0\). Then
\[
{\rm D}^+(fh)(\nabla g)\leq f\,{\rm D}^+h(\nabla g)+h\,{\rm D}^+f(\nabla g).
\]
\end{proposition}
\begin{proof}
Fix \(x\in\X\). First of all, assume that \(s(x)\coloneqq f(x)+h(x)>0\). For any \(\varepsilon>0\), it holds that
\[
g+\varepsilon f(x)h+\varepsilon h(x)f=(g+\varepsilon fh)-\underset{\eqqcolon a_{\varepsilon,x}}{\underbrace{\varepsilon(f-f(x))(h-h(x))}}+\varepsilon f(x) h(x).
\]
As \(\lip(a_{\varepsilon,x})(x)=0\) by \eqref{eq:calc_rul_lip_1} and \eqref{eq:calc_rul_lip_3}, we get \(\lip(g+\varepsilon f(x)h+\varepsilon h(x)f)(x)=\lip(g+\varepsilon fh)(x)\).
Moreover, we have that
\[\begin{split}
\lip(g+\varepsilon f(x)h+\varepsilon h(x)f)(x)&=\lip\bigg(\frac{f(x)}{s(x)}(g+\varepsilon s(x)h)+\frac{h(x)}{s(x)}(g+\varepsilon s(x)f)\bigg)(x)\\
&\leq\frac{f(x)}{s(x)}\lip(g+\varepsilon s(x)h)(x)+\frac{h(x)}{s(x)}\lip(g+\varepsilon s(x)f)(x)
\end{split}\]
by \eqref{eq:calc_rul_lip_1} and \eqref{eq:calc_rul_lip_2}. Since the function \(\R\ni t\mapsto t^2\) is convex, we deduce that
\[\begin{split}
\lip(g+\varepsilon f(x)h+\varepsilon h(x)f)^2(x)&\leq\bigg(\frac{f(x)}{s(x)}\lip(g+\varepsilon s(x)h)(x)+\frac{h(x)}{s(x)}\lip(g+\varepsilon s(x)f)(x)\bigg)^2\\
&\leq\frac{f(x)}{s(x)}\lip(g+\varepsilon s(x)h)^2(x)+\frac{h(x)}{s(x)}\lip(g+\varepsilon s(x)f)^2(x).
\end{split}\]
Therefore, we can compute
\[\begin{split}
&{\rm D}^+(fh)(\nabla g)(x)=\lim_{\varepsilon\searrow 0}\frac{\lip(g+\varepsilon f(x)h+\varepsilon h(x)f)^2(x)-\lip(g)^2(x)}{2\varepsilon}\\
\leq\,&f(x)\lim_{\varepsilon\searrow 0}\frac{\lip(g+\varepsilon s(x)h)^2(x)-\lip(g)^2(x)}{2\varepsilon s(x)}
+h(x)\lim_{\varepsilon\searrow 0}\frac{\lip(g+\varepsilon s(x)f)^2(x)-\lip(g)^2(x)}{2\varepsilon s(x)}\\
=\,&f(x)\,{\rm D}^+h(\nabla g)(x)+h(x)\,{\rm D}^+f(\nabla g)(x).
\end{split}\]
Finally, assume that \(s(x)=0\), so that \(f(x)=h(x)=0\). Then
\[\begin{split}
|\lip(g+\varepsilon fh)(x)-\lip(g)(x)|&\overset{\eqref{eq:calc_rul_lip_2}}\leq\lip(\varepsilon fh)(x)\overset{\eqref{eq:calc_rul_lip_1}}=\varepsilon\,\lip(fh)(x)\\
&\overset{\eqref{eq:calc_rul_lip_3}}\leq\varepsilon|f(x)|\lip(h)(x)+\varepsilon|h(x)|\lip(f)(x)=0,
\end{split}\]
whence it follows that \({\rm D}^+(fh)(\nabla g)(x)=0=f(x)\,{\rm D}^+h(\nabla g)(x)+h(x)\,{\rm D}^+f(\nabla g)(x)\).
\end{proof}
\subsection{Metric measure spaces and upper Laplacian bounds}
In this paper, by a \textbf{metric measure space} \((\X,\sfd,\mm)\) we mean a complete and separable metric space
\((\X,\sfd)\) equipped with a Borel measure \(\mm\geq 0\) that is finite on all bounded Borel sets. Given finite non-negative Borel measures
\((\mu_n)_{n\in\N}\) and \(\mu\) on \(\X\), we say that \(\mu_n\) \textbf{weakly converges} to \(\mu\), and we write \(\mu_n\rightharpoonup\mu\),
provided it holds that
\[
\int f\,\d\mu=\lim_{n\to\infty}\int f\,\d\mu_n\quad\text{ for every bounded continuous function }f\colon\X\to\R.
\]
Let us now introduce our notion of upper Laplacian bound.
\begin{definition}[Upper Laplacian bound]\label{def:Lapl_bound}
Let \((\X,\sfd,\mm)\) be a metric measure space. Let \(\Omega\subseteq\X\) be an open set and \(\mu\geq 0\) a Borel measure on \(\X\). Let \(g\in\LIP(\Omega)\)
be given. Then we write
\[
\Delta g\leq\mu\quad\text{ on }\Omega
\]
provided it holds that
\begin{equation}\label{eq:def_Delta}
-\int_\Omega{\rm D}^+f(\nabla g)\,\d\mm\leq\int_\Omega f\,\d\mu\quad\text{ for every }f\in\LIP_{bs}(\Omega)\text{ with }f\geq 0.
\end{equation}
\end{definition}

The support of \({\rm D}^+f(\nabla g)\) is contained in \({\rm spt}(f)\), so that \(|{\rm D}^+f(\nabla g)|\leq\1_{{\rm spt}(f)}(\Lip(g)^2+\Lip(f)^2)\)
holds on \(\Omega\) by Lemma \ref{lem:tech_results_Df(nabla_g)} ii), and thus the integral in the left-hand side of \eqref{eq:def_Delta} is well defined.
\begin{remark}[Consistency with \cite{Cav:Mon:20}]\label{rmk:consist_Lapl}{\rm
Assume \((\X,\sfd)\) is proper. Let \(g\in\LIP(\X)\) be in the domain \(D(\boldsymbol\Delta,\Omega)\) of the Laplacian of \(\Omega\)
and let \(T\) be a Radon functional over \(\Omega\) belonging to \(\boldsymbol\Delta g\llcorner_\Omega\) (in the sense of \cite[Definition 2.12]{Cav:Mon:20}).
Assume also that a Borel measure \(\mu\geq 0\) on \(\X\) satisfies \(T(h)\leq\int h\,\d\mu\) for all \(h\in\LIP_{bs}(\X)^+\). It is then clear that
\(\Delta g\leq\mu\) on \(\Omega\) (in the sense of Definition \ref{def:Lapl_bound}).
\fr}\end{remark}
\subsection{Sets of finite perimeter}
The ensuing discussion on BV functions and sets of finite perimeter is taken from \cite{Mir:03} (and \cite{Amb:DiMa:14}).
\begin{definition}[Function of bounded variation]\label{def:BV}
Let \((\X,\sfd,\mm)\) be a metric measure space. Given any function \(f\in L^1_{loc}(\mm)\) and any open set
\(\Omega\subseteq\X\), we define the \textbf{total variation} of \(f\) on \(\Omega\) as
\[
|{\bf D}f|(\Omega)\coloneqq\inf\bigg\{\liminf_{n\to\infty}\int_\Omega\lip(f_n)\,\d\mm\;\bigg|\;
(f_n)_n\subseteq\LIP_{loc}(\Omega),\,f_n\to f\text{ in }L^1_{loc}(\mm|_\Omega)\bigg\}.
\]
Then we say that \(f\) is a \textbf{function of bounded variation} provided \(|{\bf D}f|(\X)<+\infty\).
\end{definition}

When \(f\in L^1_{loc}(\mm)\) is a function of bounded variation, the set function \(\Omega\mapsto|{\bf D}f|(\Omega)\)
introduced in Definition \ref{def:BV} can be uniquely extended to a finite Borel measure, which we still denote by \(|{\bf D}f|\).
It can be readily checked that every \(f\in\LIP_{bs}(\X)\) is a function of bounded variation with
\begin{equation}\label{eq:ineq_Lip_BV}
|{\bf D}f|\leq\lip(f)\mm.
\end{equation}
Furthermore, it follows from the equivalence results in \cite{Amb:DiMa:14} that for any function \(f\in L^1(\mm)\) of bounded
variation it holds that
\begin{equation}\label{eq:def_BV_improved}
|{\bf D}f|(\X)=\inf\bigg\{\liminf_{n\to\infty}\int\lip(f_n)\,\d\mm\;\bigg|\;
(f_n)_n\subseteq\LIP_{bs}(\X),\,f_n\to f\text{ in }L^1(\mm)\bigg\}.
\end{equation}
\begin{definition}[Set of finite perimeter]\label{def:fin_per}
Let \((\X,\sfd,\mm)\) be a metric measure space. Let \(E\subseteq\X\) be Borel. Then we say that \(E\)
is a \textbf{set of finite perimeter} if \(\1_E\) is a function of bounded variation. We call
\(\Per(E;\cdot)\coloneqq|{\bf D}\1_E|\) its \textbf{perimeter measure} and \(\Per(E)\coloneqq\Per(E;\X)\) its \textbf{perimeter}.
\end{definition}

We collect below some basic properties of perimeter measures:
\begin{itemize}
\item \textbf{Complementation.} If \(E\subseteq\X\) is a set of finite perimeter, then \(\X\setminus E\) is a set of finite perimeter as well
and it holds that \(\Per(\X\setminus E;\cdot)=\Per(E;\cdot)\).
\item \textbf{Invariance under a.e.\ modifications.} If \(E\subseteq\X\) is a set of finite perimeter and \(F\subseteq\X\) is a Borel set satisfying
\(\mm(E\Delta F)=0\), then \(F\) is a set of finite perimeter as well and it holds that \(\Per(F;\cdot)=\Per(E;\cdot)\).
\item \textbf{Lower semicontinuity of perimeters.} Let \((E_n)_{n\in\N}\) be a sequence of sets of finite perimeter in \(\X\) such that
\(\sup_{n\in\N}\Per(E_n)<+\infty\). Let \(E\subseteq\X\) be a Borel set. Assume that for any \(x\in\X\) there exists \(r_x>0\) such that
\(\mm((E_n\Delta E)\cap B(x,r_x))\to 0\) as \(n\to\infty\). Then \(E\) is a set of finite perimeter and \(\Per(E;\Omega)\leq\liminf_n\Per(E_n;\Omega)\)
for every open set \(\Omega\subseteq\X\).
\end{itemize}

The following Fleming--Rishel coarea formula for metric measure spaces was proved in \cite{Mir:03}.
\begin{theorem}[Coarea formula]
Let \((\X,\sfd,\mm)\) be a metric measure space. Let \(f\in L^1_{loc}(\mm)\) be a function of bounded variation
and \(g\colon\X\to[0,+\infty)\) a Borel function. Then \(\{f<t\}\) is a set of finite perimeter for
a.e.\ \(t\in\R\), the a.e.\ defined function \(\R\ni t\mapsto\int g\,\d\Per(\{f<t\};\cdot)\) has a Borel
representative, and it holds that
\begin{equation}\label{eq:coarea}
\int g\,\d|{\bf D}f|=\int_{-\infty}^{+\infty}\!\!\!\int g\,\d\Per(\{f<t\};\cdot)\,\d t.
\end{equation}
\end{theorem}

Finally, let us recall the concept of isoperimetric set in a metric measure space.
\begin{definition}[Isoperimetric set]\label{def:isoper_set}
Let \((\X,\sfd,\mm)\) be a metric measure space. Let \(E\subseteq\X\) be a set of finite perimeter with \(0<\mm(E)<+\infty\).
Then we say that \(E\) is an \textbf{isoperimetric set} provided
\[
\Per(E)=\inf\Big\{\Per(F)\;\Big|\;F\subseteq\X\text{ of finite perimeter, }\mm(F)=\mm(E)\Big\}.
\]
\end{definition}
\section{Proof of Theorem \ref{thm:main_1}}\label{s:main_1}
In this section, we prove that Theorem \ref{thm:main_1} holds through an explicit construction of an example.
\medskip

We take as the metric space the \(2\)-dimensional flat torus $\X \coloneqq \mathbb T = \mathbb R^2/ \mathbb Z^2$,
equipped with its geodesic distance $\sfd$. To define the reference measure, we first define a sequence
$(x_i)_{i=1}^\infty \subseteq\X$. We start by fixing an arbitrary $x_1 \in \X$. Then we define the rest of the sequence
of points $x_i$ iteratively as follows. Suppose $(x_i)_{i=1}^k$ have been selected. Now, let $x_{k+1}$ be one of
the points maximising the distance to $\bigcup_{i=1}^k B(x_i,2^{-i-2})$ in $\X$. Let us abbreviate
$r_i\coloneqq 2^{-i-2}$, $B_i\coloneqq B(x_i,r_i)$ and define $U \coloneqq \bigcup_{i=1}^\infty B_i$.
We define the reference measure $\mm$ in $\X$ to be  $\rho\mathcal H^2$ with $\rho = c \coloneqq 2^{-10}\pi^{-2}$
on $U$ and $\rho \coloneqq 1$ on $\X \setminus U$. Finally, we denote $V \coloneqq \X \setminus U$.
\begin{theorem}\label{thm:example}
The metric measure space \((\X,\sfd,\mm)\) as above is an infinitesimally Hilbertian, Ahlfors $2$-regular PI space
with \((\X,\sfd)\) compact and geodesic. Moreover, $V\subseteq\X$ is an isoperimetric set with the property that
any \(\mm\)-a.e.\ representative \(\tilde V\) of \(V\) has empty topological interior.
\end{theorem}

We recall that `infinitesimally Hilbertian' means that the \(2\)-Sobolev space \(W^{1,2}(\X,\sfd,\mm)\) is a Hilbert space \cite{Gig:15}.
The verification of Theorem \ref{thm:example} will be subdivided into several steps. Let us begin with an observation on the construction.
\begin{lemma}\label{lma:density_separation}
The following properties hold:
\begin{itemize}
    \item[\(\rm i)\)] The set $U$ is dense in $V$.
    \item[\(\rm ii)\)] For any $k \ge 2$, we have that
    \[
    \dist\left(x_k,\bigcup_{i=1}^{k-1} B_i\right) \ge 2r_k.
    \]
\end{itemize}
\end{lemma}
\begin{proof}
Suppose the claim i) is not true. Then there exist $x \in\X$ and $\varepsilon > 0$ so that $B(x,\varepsilon) \subseteq V$.
By the construction of the sequence $(x_i)_i$, this means in particular that for every $k>j$ we have
\[
\sfd(x_k,x_j) \ge \dist\left(x_k,\bigcup_{i=1}^{k-1} B_i\right) \ge \dist\left(x,\bigcup_{i=1}^{k-1} B_i\right) \ge \dist(x,U) \ge \varepsilon.
\]
This contradicts the fact that $(\X,\sfd)$ is compact. So, i) is proven.

Let us then prove ii). We argue by considering the distances $\sfd(x_1,x_i)$ for \(i=1,\ldots,k-1\).
The intervals $[\sfd(x_1,x_i)-r_i,\sfd(x_1,x_i)+r_i]$ with $i$ from $1$ to $k-1$ cover at most length
\[
\sum_{i=1}^{k-1}2r_i = \sum_{i=1}^{k-1}2^{-i-1}<\frac12
\]
of the interval $[0,1]$. Thus, the set
\[
[0,1] \setminus \bigcup_{i=1}^{k-1}[\sfd(x_1,x_i)-r_i,\sfd(x_1,x_i)+r_i]
\]
contains at least one interval $(a,b)$ of length $\frac{1}{2k} \ge 2^{-k} = 4 r_k$.
Thus, any given point $y \in\X$ with $\sfd(x_1,y) = \frac{a+b}{2}$ satisfies
\[
    \dist\left(y,\bigcup_{i=1}^{k-1} B_i\right) \ge 2r_k.
\]
Consequently, the chosen point $x_k$ also satisfies the desired inequality.
\end{proof}

Let us now start working towards showing that $V$ is an isoperimetric set.
For this, we define a function $\varphi\colon [0,\mm(U)] \to \mathbb R$ in the following way. Given any $t \in (0,\mm(U))$,
let $k_t \in \mathbb N$ be the index for which $\sum_{i=1}^{k_t}c\pi r_i^2 < t \le \sum_{i=1}^{k_t+1}c\pi r_i^2$. Now, set
\[
\varphi(t) \coloneqq \sum_{i=1}^{k_t}2c\pi r_i + 2c\pi\sqrt{\frac{t}{c\pi} - \sum_{i=1}^{k_t} r_i^2}.
\]
We extend $\varphi$ continuously to the endpoints $0$ and $\mm(U)$, that is $\varphi(0)\coloneqq 0$ and
\[
\varphi(\mm(U)) \coloneqq \sum_{i=1}^\infty 2c\pi r_i = \sum_{i=1}^\infty 2c\pi 2^{-i-2} = \frac{c\pi}{2} = \Per(U).
\]
\begin{lemma}\label{lma:phi}
 For every $t \in (0,\mm(U))$, we have that
 \[
 \varphi(\mm(U)) < \varphi(t) + \frac{1}{4\sqrt{\pi}}\sqrt{\mm(U)-t}.
 \]
\end{lemma}
\begin{proof}
 Let $k_t \in \mathbb N$ be as in the definition of $\varphi(t)$. By the definition of $\varphi$, we have
 \begin{align*}
 \varphi(\mm(U)) - \varphi(t) & = \sum_{i=k_t+1}^{\infty}2c\pi r_i - 2c\pi\sqrt{\frac{t}{c\pi} - \sum_{i=1}^{k_t} r_i^2}
 \le \sum_{i=k_t+1}^{\infty}2c\pi r_i \\
 & = \sum_{i=k_t+1}^{\infty}c\pi 2^{-i-1} 
 = c\pi2^{-k_t-1} = 8c\pi r_{k_t+2}.
 \end{align*}
 Since $\mm(U) - t > c\pi r_{k_t+2}^2$,
 we have
 \[
  \varphi(\mm(U)) - \varphi(t) \leq  8\pi cr_{k_t+2} <  8\sqrt{\pi c} \sqrt{\mm(U)-t}  = \frac{1}{4\sqrt{\pi}}\sqrt{\mm(U)-t},
 \]
 where the last equality comes from the choice of $c$.
\end{proof}

Let us now take a set $E \subseteq \X$ of finite perimeter to be a competitor for minimising the perimeter under the constraint
$\mm(E) = \mm(U)$. Our aim is to now prove that $\Per(U) \le \Per(E)$.
\begin{lemma}\label{lma:1}
It holds that $\Per(E\cap U) \ge \varphi(\mm(E\cap U))$.
\end{lemma}
\begin{proof}
If \(\mm(E\cap U)=\mm(U)\), then \(E\cap U=U\) up to \(\mm\)-null sets and thus \(\Per(E\cap U)=\varphi(\mm(E\cap U))\).
Next, assume that \(\mm(E\cap U)<\mm(U)\). We define the open balls \(\tilde E_i\) for \(i\geq 1\) as
\[
\tilde E_i\coloneqq B\bigg(x_i,\sqrt{\frac{\mm(E\cap B_i)}{c\pi}}\bigg).
\]
Note that the set \(\tilde E\coloneqq\bigcup_{i=1}^\infty\tilde E_i\) satisfies \(\mm(\tilde E)=\mm(E\cap U)\) and \(\Per(\tilde E)\leq\Per(E\cap U)\) thanks to the
Euclidean isoperimetric inequality. Since \(\mm(\tilde E)<\mm(U)\), we have that \(\tilde E_i\neq B_i\) for some \(i\geq 1\). Denote by \(k\geq 1\) the smallest such
index \(i\). Since \(\sum_{i=k}^\infty\mm(\tilde E_i)<+\infty\), we can find \(m>k\) such that \(\sum_{i=m}^\infty\mm(\tilde E_i)<\mm(B_k)-\mm(\tilde E_k)\). Then there
exists a (unique) radius \(\bar r>0\) such that \(\tilde E_k\subseteq B(x_k,\bar r)\subseteq B_k\) and \(\mm(B(x_k,\bar r))=\mm(\tilde E_k)+\sum_{i=m}^\infty\mm(\tilde E_i)\),
and the Euclidean isoperimetric inequality ensures that \(\Per(B(x_k,\bar r))\leq\Per(\tilde E_k)+\sum_{i=m}^\infty\Per(\tilde E_i)\). Therefore, letting \(\tilde F_i\coloneqq\tilde E_i\)
for every \(i\in\{1,\ldots,m\}\setminus\{k\}\) and \(\tilde F_k\coloneqq B(x_k,\bar r)\), we have that \(\mm(\tilde F_1\cup\dots\cup\tilde F_m)=\mm(E\cap U)\) and
\(\Per(\tilde F_1\cup\dots\cup\tilde F_m)\leq\Per(E\cap U)\). For some bijective function \(\sigma\colon\{1,\ldots,m\}\to\{1,\ldots,m\}\), it holds that
\(\mm(\tilde F_{\sigma(1)})\geq\mm(\tilde F_{\sigma(2)})\geq\dots\geq\mm(\tilde F_{\sigma(m)})\). Since \(\mm(B_1)>\mm(B_2)>\dots>\mm(B_m)\), it can be readily
checked that \(\mm(\tilde F_{\sigma(i)})\leq\mm(B_i)\) for every \(i=1,\ldots,m\), so that accordingly
\[
F_i\coloneqq B\big(x_i,{\rm rad}(\tilde F_{\sigma(i)})\big)\subseteq B_i\quad\text{ for every }i=1,\ldots,m,
\]
where \({\rm rad}(\tilde F_{\sigma(i)})>0\) denotes the radius of the ball \(\tilde F_{\sigma(i)}\). Note that \(\mm(F_1\cup\dots\cup F_m)=\mm(E\cap U)\)
and \(\Per(F_1\cup\dots\cup F_m)\leq\Per(E\cap U)\). Letting \(\Per_{\rm euc}\) be the Euclidean perimeter in \(\R^2\), we have
\[
\Per_{\rm euc}(B(x,R+t))+\Per_{\rm euc}(B(y,r-s))<\Per_{\rm euc}(B(x,R))+\Per_{\rm euc}(B(y,r))
\]
whenever \(x,y\in\R^2\), \(0<t<s<r\leq R\) and \(\mathcal H^2(B(x,R+t)\setminus B(x,R))=\mathcal H^2(B(y,r)\setminus B(y,r-s))\).
By repeatedly making use of this fact, we deduce that for some \(j\in\{1,\ldots,m\}\) and \(\bar s<r_j\) one has
\[\begin{split}
\mm\big(B_1\cup\dots\cup B_{j-1}\cup B(x_k,\bar s)\big)&=\mm(F_1\cup\dots\cup F_m)=\mm(E\cap U),\\
\Per\big(B_1\cup\dots\cup B_{j-1}\cup B(x_k,\bar s)\big)&\leq\Per(F_1\cup\dots\cup F_m)\leq\Per(E\cap U).
\end{split}\]
In particular, \(\varphi(\mm(E\cap U))=\Per\big(B_1\cup\dots\cup B_{j-1}\cup B(x_k,\bar s)\big)\leq\Per(E\cap U)\), as desired.
\end{proof}

Let us introduce another abbreviation, $U_c \coloneqq \bigcup_{i=1}^\infty\overline{B}_i$.

\begin{lemma}\label{lma:2}
Suppose that $\Per(E) < \Per(U)$. Then
 we have 
 \[
 \Per(E\setminus U; \X \setminus U_c) \ge \frac1{4\pi}  \Per_{\rm euc}(E\setminus U; U_c) = \frac{1}{4\pi c}\Per(E\setminus U; U_c).
 \]
\end{lemma}
\begin{proof}
Up to an \(\mm\)-a.e.\ modification of the set \(E\), we can assume that \(\diam(E)\leq\frac{1}{2}\Per(E)\).
By the decomposition theorem for sets of finite perimeter into indecomposable sets \cite{AmbCasMasMor:01}
(see also \cite{BonPasRaj:20,Lah:23}), it suffices to prove that
\[
\Per(F;\X\setminus U_c)\geq\frac{1}{4\pi c}\Per(F;U_c)\quad\text{ for every indecomposable component }F\text{ of }E\setminus U.
\]
To this aim, fix an indecomposable component \(F\).
By assumption, $\Per(E) < \Per(U) = \frac{c\pi}{2}$. Thus,
\begin{equation}\label{eq:diambound}
{\rm diam}(F)\leq\diam(E\setminus U)\leq\diam(E)\leq \frac{\Per(E)}{2}<\frac{c\pi}{4}< \frac14.
\end{equation}
 Without loss of generality, we may assume that $\Per_{\rm euc}(F; U_c) > 0$.
 Let $k \in \mathbb N$ be the smallest integer so that 
 $\Per(F; \overline{B}_k)>0$. By \eqref{eq:diambound}, we have $\overline F\subseteq B(x_k,\frac12)$.
 We consider two cases and
  suppose first that $\Per(F; \overline{B}_i) = 0$ for all $i>k$.
 Let $p \colon B(x_k,\frac12)\setminus B_k \to \partial B_k$ be the radial projection towards the center $x_k$ of $B_k$. Since $p$ is $1$-Lipschitz, we have
  \begin{align*}
  \Per\left(F;  U_c\right) & = \Per(F; \overline{B}_k)
  = c\,\mathcal H^1(\partial^e F\cap \partial B_k) \\
  & \le c\,\mathcal H^1\big(p(\partial^e F\cap ( B(x_k,1/2) \setminus \partial B_k))\big) \\
  & \le c\,\mathcal H^1\big(\partial^e F\cap ( B(x_k,1/2) \setminus \partial B_k)\big)
  = c\,\Per\left(F; \X \setminus U_c\right).
  \end{align*}
Suppose then that $\Per(F; \overline{B}_i) > 0$ for some $i>k$. We let $m \in \mathbb N$ be the second smallest integer so that $\Per(F; \overline{B}_m)>0$.
We divide this second case into two subcases.
If $\Per(F; \overline{B}_k) < 4\pi cr_m$, then
\begin{equation}\label{eq:subcase1_1}
 \Per(F; U_c) = \Per(F; \overline{B}_k) + \sum_{i=m}^\infty\Per(F;\overline{B}_i)
 <  4\pi cr_m + \sum_{i=m}^\infty 2\pi c r_i = 8\pi cr_m.
\end{equation}
By Lemma \ref{lma:density_separation} ii), we have
\begin{equation}\label{eq:secondcase}
 \Per\left(F; \X \setminus U_c\right) \ge 2r_m.
\end{equation}
The combination of \eqref{eq:subcase1_1} and \eqref{eq:secondcase} then gives the desired bound
\[
 \Per\left(F; \X \setminus U_c\right) \ge 2r_m = \frac{8\pi cr_m}{4\pi c} > \frac1{4\pi c}\Per(F; U_c).
\]
Finally, assume the last subcase $\Per(F; \overline{B}_k) \ge 4\pi cr_m$. Then
\begin{equation}\label{eq:subcase2_1}
\begin{split}
\Per(F; U_c) &= \Per(F; \overline{B}_k) + \sum_{i=m}^\infty \Per(F; \overline{B}_i)
\le \Per(F; \overline{B}_k) + \sum_{i=m}^\infty 2\pi c r_i\\
&= \Per(F; \overline{B}_k) + 4\pi cr_m
\le 2\,\Per(F; \overline{B}_k).
\end{split}
\end{equation}
 By considering again the radial projection $p\colon B(x_k,\frac{1}{2})\setminus B_k\to\partial B_k$ towards $x_k$, we first note that
\[
 \mathcal H^1\left(p\left(\bigcup_{i=m}^\infty B(x_k,1/2)\cap\overline{B}_i\right)\right)\leq\sum_{i=m}^\infty 2r_i\le
 \sum_{i=m}^\infty\pi r_i = 2\pi r_m \le \frac{1}{2c}\Per(F; \overline{B}_k).
\]
Thus, by using this and \eqref{eq:subcase2_1}, we get
\begin{align*}
\Per\left(F; \X \setminus U_c\right) & 
= \mathcal H^1(\partial^e F \cap (\X \setminus U_c))
\ge \mathcal H^1\big(p(\partial^e F \cap (\X \setminus U_c))\big)\\
&\ge \mathcal H^1(\partial^e F \cap \partial B_k) - \mathcal H^1\left(p\left(\bigcup_{i=m}^\infty B(x_k,1/2)\cap\overline{B}_i\right)\right)\\
&\ge \frac1c\Per(F; \overline{B}_k) - \frac{1}{2c}\Per(F; \overline{B}_k)
= \frac1{2c}\Per(F; \overline{B}_k)\\
&\geq \frac1{4c}\Per(F; U_c) >  \frac1{4\pi c}\Per(F; U_c).
\end{align*}
The proof is complete.
\end{proof}
\begin{lemma}\label{lma:3}
 We have $\Per(E\setminus U; \X \setminus U_c) \ge \frac{2\sqrt\pi}{4\pi+1}\sqrt{\mm(E\setminus U)}$.
\end{lemma}
\begin{proof}
By Lemma \ref{lma:2} and the Euclidean isoperimetric inequality, we have
\[\begin{split}
2\sqrt\pi\sqrt{\mm(E\setminus U)}&\leq\Per_{\rm euc}(E\setminus U)=
\Per_{\rm euc}(E\setminus U;U_c)+\Per_{\rm euc}(E\setminus U;\X\setminus U_c)\\
&\leq(4\pi+1)\Per(E\setminus U;\X\setminus U_c),
\end{split}\]
whence the claim follows.
\end{proof}

We are now ready to complete the proof of Theorem \ref{thm:example}.
\begin{proof}[Proof of Theorem \ref{thm:example}]
Since $(\X,\sfd,\mm)$ differs from the standard flat torus by only the weight $\rho$ that is bounded away from zero and bounded from above, $(\X,\sfd,\mm)$ is complete, geodesic,
it satisfies a $(1,1)$-Poincar\'e inequality and it is Ahlfors $2$-regular. The space $(\X,\sfd,\mm)$ is also easily seen to be infinitesimally Hilbertian; this follows for instance
from \cite{LucPasRaj21}. As a consequence of Lemma \ref{lma:density_separation} i), we have that $V$ has no representative with a topological interior point. Finally, for any isoperimetric competitor $E$ for $U$, we can estimate
\begin{align*}
\Per(E) & = \Per(E\cap U) + \Per(E\setminus U) - 2\Per(E\setminus U; U_c)\\
 & \ge \Per(E\cap U) + \Per(E\setminus U; \X\setminus U_c) - \Per(E\setminus U; U_c)\\
  & \ge \varphi(\mm(E\cap U)) + \left(1-4\pi c\right)\Per(E\setminus U; \X \setminus U_c)\\
  & \ge \varphi(\mm(E\cap U)) + 2\sqrt\pi\frac{1-4\pi c}{4\pi+1}\sqrt{\mm(E\setminus U)}\\
  & \ge \varphi(\mm(E\cap U)) + \frac{1}{4\sqrt{\pi}}\sqrt{\mm(E\setminus U)}\\
  & > \varphi(\mm(U)) = \Per(U), 
\end{align*}
where the third inequality follows from Lemmata \ref{lma:1} and \ref{lma:2}, the fourth from Lemma \ref{lma:3}, the final one from Lemma \ref{lma:phi}.
This shows that \(U\) is an isoperimetric set. As \(\Per(F)=\Per(\X\setminus F)\) for any \(F\subseteq\X\) of finite perimeter, and \(\mm(V)<+\infty\),
we conclude that \(V\) is an isoperimetric set.
\end{proof}
\section{Proof of Theorems \ref{thm:main_2} and \ref{thm:MCP}}\label{s:main_2}
Given a metric space \((\X,\sfd)\) and \(x\in\X\), we recall that by \(\sfd_x\) we denote the \(1\)-Lipschitz function
\[
\sfd_x(y)\coloneqq\sfd(x,y)\quad\text{ for every }y\in\X.
\]
Before passing to the verification of Theorem \ref{thm:main_2}, we need to prove a couple of auxiliary results.
\begin{lemma}\label{lem:tech_results_BV}
Let \((\X,\sfd,\mm)\) be a metric measure space. Then the following properties hold:
\begin{itemize}
\item[\(\rm i)\)] Given any \(x\in\X\), we have that \(\mm(S(x,r))=0\) for all but countably many \(r>0\).
\item[\(\rm ii)\)] Given any \(E\subseteq\X\) of finite perimeter, \(\Per(E;S(x,r))=0\) for all but countably many \(r>0\).
\item[\(\rm iii)\)] Given any \(f\in L^1(\mm)^+\cap L^\infty(\mm)\) of bounded variation, there exists \((f_n)_n\subseteq\LIP_{bs}(\X)\)
such that \(0\leq f_n\leq\|f\|_{L^\infty(\mm)}\) for every \(n\in\N\), \(f_n\to f\) in \(L^1(\mm)\) and
\(\int\lip(f_n)\,\d\mm\to|{\bf D}f|(\X)\).
\item[\(\rm iv)\)] If \(f\in L^1_{loc}(\mm)\) is a function of bounded variation, and a sequence \((f_n)_n\subseteq\LIP_{loc}(\X)\)
satisfies \(f_n\to f\) in \(L^1_{loc}(\mm)\) and \(\int\lip(f_n)\,\d\mm\to|{\bf D}f|(\X)\), then we have that \(\lip(f_n)\mm\rightharpoonup|{\bf D}f|\).
\item[\(\rm v)\)] Let \(f_n\colon\X\to\R\), for \(n\in\N\), and \(f\colon\X\to\R\) be given Borel functions
such that \(\int|f_n-f|\,\d\mm\to 0\) as \(n\to\infty\). Fix any \(x\in\X\). Then there exists a subsequence \((n_k)_k\) such that
\[
\lim_{k\to\infty}\int|f_{n_k}-f|\,\d\Per(B(x,r);\cdot)=0\quad\text{ for a.e.\ }r>0.
\]
\end{itemize}
\end{lemma}
\begin{proof}
\ \\
{\bf i), ii).} They follow from a well-known fact in measure theory: if \((\Y,\Sigma,\mu)\) is a \(\sigma\)-finite measure space
and \(\{E_i\}_{i\in I}\subseteq\Sigma\) are pairwise disjoint sets, then \(\mu(E_i)=0\) for all but countably many \(i\in I\).\\
{\bf iii).} By \eqref{eq:def_BV_improved}, there is \((\tilde f_n)_n\subseteq\LIP_{bs}(\X)\) such that \(\tilde f_n\to f\) in \(L^1(\mm)\) and \(\int\lip(\tilde f_n)\,\d\mm\to|{\bf D}f|(\X)\).
Define \(f_n\coloneqq(\tilde f_n\wedge\|f\|_{L^\infty(\mm)})\vee 0\in\LIP_{bs}(\X)\) for every \(n\in\N\). Since \(f_n\to f\) in \(L^1(\mm)\)
as \(n\to\infty\) and \(\lip(f_n)\leq\lip(\tilde f_n)\) for every \(n\in\N\), we can conclude that \(\int\lip(f_n)\,\d\mm\to|{\bf D}f|(\X)\).\\
{\bf iv).} For any \(\Omega\subseteq\X\) open, we have \(|{\bf D}f|(\Omega)\leq\liminf_n\int_\Omega\lip(f_n)\,\d\mm=\liminf_n(\lip(f_n)\mm)(\Omega)\).
We also assumed \((\lip(f_n)\mm)(\X)\to|{\bf D}f|(\X)\), thus the Portmanteau theorem yields \(\lip(f_n)\mm\rightharpoonup|{\bf D}f|\).\\
{\bf v).} Fix any \(R\in\N\). Take a compactly-supported \(1\)-Lipschitz function \(\psi\colon\R\to[0,1]\) with \(\psi=1\) on \([0,R]\).
Then \(\psi\circ\sfd_x\in\LIP_{bs}(\X)\) and \(\psi\circ\sfd_x=\sfd_x\) on \(B(x,R)\), thus the coarea formula \eqref{eq:coarea} gives
\[\begin{split}
\int_0^R\!\!\!\int|f_n-f|\,\d\Per(B(x,r);\cdot)\,\d r&=\int_0^R\!\!\!\int\1_{B(x,R)}|f_n-f|\,\d\Per(\{\psi\circ\sfd_x<r\};\cdot)\,\d r\\
&\leq\int|f_n-f|\,\d|{\bf D}(\psi\circ\sfd_x)|\overset{\eqref{eq:ineq_Lip_BV}}\leq\int|f_n-f|\lip(\psi\circ\sfd_x)\,\d\mm\\
&\leq\int|f_n-f|\,\d\mm\to 0\quad\text{ as }n\to\infty.
\end{split}\]
Therefore, we can extract a subsequence \((n_k)_k\) such that \(\lim_k\int|f_{n_k}-f|\,\d\Per(B(x,r);\cdot)=0\) for a.e.\ \(r\in(0,R)\).
Since \(R\in\N\) is arbitrary, the claim follows by a diagonalisation argument.
\end{proof}
\begin{proposition}\label{prop:Minkowski_per}
Let \((\X,\sfd,\mm)\) be a metric measure space with \((\X,\sfd)\) proper. Fix any \(x\in\X\). Given any \(k\in\N\) and \(r>0\),
we define the auxiliary \(k\)-Lipschitz function \(\phi_k^r\colon\R\to[0,1]\) as
\[
\phi_k^r(t)\coloneqq\left\{\begin{array}{lll}
1\\
1-k(t-r)\\
0
\end{array}\quad\begin{array}{lll}
\text{ for every }t\in(-\infty,r],\\
\text{ for every }t\in(r,r+k^{-1}),\\
\text{ for every }t\in[r+k^{-1},+\infty).
\end{array}\right.
\]
Moreover, let us define \(g_k^r\coloneqq\phi_k^r\circ\sfd_x\in\LIP_{bs}(\X)\). Then for a.e.\ \(r>0\) it holds that
\(g_k^r\to\1_{B(x,r)}\) in \(L^1(\mm)\) as \(k\to\infty\) and that the set \(\{\lip(g_k^r)\mm\}_{k\in\N}\) is weakly relatively compact.
\end{proposition}
\begin{proof}
Define \(m\colon(0,+\infty)\to\R\) as \(m(r)\coloneqq\mm(B(x,r))\) for every \(r>0\). For a.e.\ \(r>0\) we have that
\[\begin{split}
\int|g_k^r-\1_{B(x,r)}|\,\d\mm&=\mm(S(x,r))+\int_{B(x,r+k^{-1})\setminus\bar B(x,r)}g_k^r\,\d\mm\\
&\leq\mm\big(B(x,r+k^{-1})\setminus B(x,r)\big)\to\mm(S(x,r))=0\quad\text{ as }k\to\infty
\end{split}\]
by Lemma \ref{lem:tech_results_BV} i). Moreover, as \(m\) is non-decreasing (and thus a.e.\ differentiable), we have that
\[\begin{split}
\limsup_{k\to\infty}\int\lip(g_k^r)\,\d\mm&\overset{\eqref{eq:calc_rul_lip_4}}\leq
\limsup_{k\to\infty}\int(\lip(\phi^r_k)\circ\sfd_x)\lip(\sfd_x)\,\d\mm
\leq\lim_{k\to\infty}\int\lip(\phi^r_k)\circ\sfd_x\,\d\mm\\
&=\lim_{k\to\infty}k\,\mm\big(\bar B(x,r+k^{-1})\setminus B(x,r)\big)=m'(r)<+\infty\quad\text{ for a.e.\ }r>0,
\end{split}\]
whence it follows that \(\sup_{k\in\N}\|\lip(g_k^r)\|_{L^1(\mm)}<+\infty\). Since \(\lip(g_k^r)\mm\) is concentrated on
\(\bar B(x,r+1)\) for all \(k\in\N\), Prokhorov's theorem implies that \(\{\lip(g_k^r)\mm\}_{k\in\N}\) is weakly relatively compact.
\end{proof}

We are now in a position to prove Theorem \ref{thm:main_2}.
\begin{proof}[Proof of Theorem \ref{thm:main_2}]
Using Lemma \ref{lem:tech_results_BV} iii), we can find a sequence \((f_n)_n\subseteq\LIP_{bs}(\X)\) with \(0\leq f_n\leq 1\)
such that \(f_n\to\1_E\) in \(L^1(\mm)\) and \(\int\lip(f_n)\,\d\mm\to\Per(E)\), whence it follows that
\(\lip(f_n)\mm\rightharpoonup\Per(E;\cdot)\) by Lemma \ref{lem:tech_results_BV} iv). Taking Lemma \ref{lem:tech_results_BV} i),
ii), v) and Proposition \ref{prop:Minkowski_per} into account, we can find a negligible set \(N\subseteq(0,R)\) and a
(non-relabelled) subsequence of \((f_n)_n\) such that the following conditions are satisfied for every \(r\in(0,R)\setminus N\):
\begin{itemize}
\item[a)] \(\mm(S(x,r))=\mm(S(x,r+k^{-1}))=0\) for every \(k\in\N\).
\item[b)] \(\Per(E;S(x,r))=0\), thus in particular \(\lim_n\int_{B(x,r)}\lip(f_n)\,\d\mm=\Per(E;B(x,r))\).
\item[c)] \(f_n\to\1_E\) in \(L^1(\Per(B(x,r);\cdot))\) as \(n\to\infty\).
\item[d)] \(g_k^r\to\1_{B(x,r)}\) in \(L^1(\mm)\) as \(k\to\infty\), where \(g^r_k\) is defined as in Proposition \ref{prop:Minkowski_per}.
\item[e)] \(\{\lip(g^r_k)\mm\}_{k\in\N}\) is weakly relatively compact.
\end{itemize}
Now, fix \(r\in(0,R)\setminus N\). Given any \(n,k\in\N\) with \(k>\frac{1}{R-r}\), Propositions \ref{prop:chain_Dg(nabla_g)} and \ref{prop:Leibniz_Dg(nabla_g)}
ensure that
\[\begin{split}
\int_{B(x,R)}{\rm D}^+(f_n g_k^r)(\nabla\sfd_x^2)\,\d\mm&=2\int_{B(x,R)}\sfd_x\,{\rm D}^+(f_n g_k^r)(\nabla\sfd_x)\,\d\mm\\
&\leq 2\int_{B(x,R)}\sfd_x\,f_n\,{\rm D}^+g_k^r(\nabla\sfd_x)\,\d\mm+2\int_{B(x,R)}\sfd_x\,g_k^r\,{\rm D}^+f_n(\nabla\sfd_x)\,\d\mm.
\end{split}\]
As \(\lip(g^r_k)=k\) on \(B(x,r+k^{-1})\setminus\bar B(x,r)\), by using a) and Lemma \ref{lem:tech_results_Df(nabla_g)} iv), v), vii) we obtain that
\[\begin{split}
\int_{B(x,R)}\sfd_x\,f_n\,{\rm D}^+g_k^r(\nabla\sfd_x)\,\d\mm&=\int_{B(x,r+k^{-1})\setminus\bar B(x,r)}\sfd_x\,f_n\,{\rm D}^+(-k\sfd_x+(1+kr))(\nabla\sfd_x)\,\d\mm\\
&=-k\int_{B(x,r+k^{-1})\setminus\bar B(x,r)}\sfd_x\,f_n\,{\rm D}^-\sfd_x(\nabla\sfd_x)\,\d\mm\\
&=-k\int_{B(x,r+k^{-1})\setminus\bar B(x,r)}\sfd_x\,f_n\,\lip(\sfd_x)^2\,\d\mm\\
&=-\int_{B(x,r+k^{-1})\setminus\bar B(x,r)}\sfd_x\,f_n\,\lip(g^r_k)\,\d\mm\leq-r\int f_n\,\lip(g^r_k)\,\d\mm.
\end{split}\]
Moreover, by applying Lemma \ref{lem:tech_results_Df(nabla_g)} vi) we can estimate
\[\begin{split}
\int_{B(x,R)}\sfd_x\,g_k^r\,{\rm D}^+f_n(\nabla\sfd_x)\,\d\mm&\leq
\bigg(r+\frac{1}{k}\bigg)\int_{B(x,r+k^{-1})}|{\rm D}^+f_n(\nabla\sfd_x)|\,\d\mm\\
&\leq\bigg(r+\frac{1}{k}\bigg)\int_{B(x,r+k^{-1})}\lip(f_n)\,\d\mm.
\end{split}\]
Since we assumed that \(\Delta\sfd_x^2\leq C\mm\) on \(B(x,R)\),
the above estimates yield
\begin{equation}\label{eq:deform_prop_aux}\begin{split}
2r\int f_n\,\lip(g_k^r)\,\d\mm-2\bigg(r+\frac{1}{k}\bigg)\int_{B(x,r+k^{-1})}\lip(f_n)\,\d\mm
&\leq-\int_{B(x,R)}{\rm D}^+(f_n g_k^r)(\nabla\sfd_x^2)\,\d\mm\\
&\leq C\int_{B(x,R)}f_n g_k^r\,\d\mm.
\end{split}\end{equation}
Since \(\{\lip(g_k^r)\mm\}_{k\in\N}\) is weakly relatively compact by e), we can extract a subsequence \((k_j)_j\)
(depending on \(r\)) and a finite Borel measure \(\nu_r\geq 0\) on \(\X\) such that \(\lip(g_{k_j}^r)\mm\rightharpoonup\nu_r\)
as \(j\to\infty\). By virtue of the fact that \(g_{k_j}^r\to\1_{B(x,r)}\) in \(L^1(\mm)\) as \(j\to\infty\) by d), we deduce that
\(\Per(B(x,r);\cdot)\leq\nu_r\). Writing \eqref{eq:deform_prop_aux} for \(k=k_j\) and letting \(j\to\infty\), we get that
\begin{equation}\label{eq:deform_prop_aux2}\begin{split}
2r\int f_n\,\d\Per(B(x,r);\cdot)-2r\int_{B(x,r)}\lip(f_n)\,\d\mm
&\leq 2r\int f_n\,\d\nu_r-2r\int_{B(x,r)}\lip(f_n)\,\d\mm\\
&\leq C\int_{B(x,r)}f_n\,\d\mm.
\end{split}\end{equation}
Thanks to c) and b), by letting \(n\to\infty\) in \eqref{eq:deform_prop_aux2} we conclude that
\[
2r\Per(B(x,r);E)-2r\Per(E;B(x,r))\leq C\mm(E\cap B(x,r)),
\]
thus obtaining \eqref{eq:def_prop_1}.

Next, let us assume in addition that \((\X,\sfd,\mm)\) is a PI space. We deduce from
Lemma \ref{lem:tech_results_BV} ii) that \(\Per\big(E^{(1)};\partial^e(\X\setminus B(x,r))\big)\leq\Per(E^{(1)};S(x,r))=0\) for a.e.\ \(r\in(0,R)\),
thus it follows from the fact that \((\X\setminus B(x,r))^{(1)}\subseteq\X\setminus B(x,r)\) and from \cite[Proposition 2.6]{Ant:Pas:Poz:Vio:23} that
\[
\Per(E\setminus B(x,r))\leq\Per(E;\X\setminus B(x,r))+\Per(B(x,r);E^{(1)})\quad\text{ for a.e.\ }r\in(0,R).
\]
Consequently, by applying \eqref{eq:def_prop_1} (with \(E\) replaced by \(E^{(1)}\)) we obtain that
\begin{equation}\label{eq:main_2_aux}\begin{split}
\Per(E\setminus B(x,r))&\leq\Per(E;\X\setminus B(x,r))+C\frac{\mm(E^{(1)}\cap B(x,r))}{2r}+\Per(E^{(1)};B(x,r))\\
&\leq\Per(E;\X\setminus B(x,r))+C\frac{\mm(E\cap B(x,r))}{2r}+\Per(E;B(x,r))\\
&=C\frac{\mm(E\cap B(x,r))}{2r}+\Per(E)\quad\text{ for a.e.\ }r\in(0,R).
\end{split}\end{equation}
Finally, fix an arbitrary radius \(r\in(0,R)\). Applying \eqref{eq:main_2_aux}, we can find an increasing sequence of radii
\((r_n)_{n\in\N}\subseteq(0,R)\) such that \(r_n\nearrow r\) and \(\Per(E\setminus B(x,r_n))\leq C(2r_n)^{-1}\mm(E\cap B(x,r_n))+\Per(E)\)
for every \(n\in\N\). Given that \(\mm\big(E\cap(B(x,r)\setminus B(x,r_n))\big)=|\mm(E\cap B(x,r))-\mm(E\cap B(x,r_n))|\to 0\) as \(n\to\infty\),
the lower semicontinuity property of perimeters ensures that
\[\begin{split}
\Per(E\setminus B(x,r))&\leq\liminf_{n\to\infty}\Per(E\setminus B(x,r_n))\leq\lim_{n\to\infty}C\frac{\mm(E\cap B(x,r_n))}{2r_n}+\Per(E)\\
&=C\frac{\mm(E\cap B(x,r))}{2r}+\Per(E),
\end{split}\]
thus proving the validity of \eqref{eq:def_prop_2}.
\end{proof}

Finally, let us conclude by proving Theorem \ref{thm:MCP}.
\begin{proof}[Proof of Theorem \ref{thm:MCP}]
Let \((\X,\sfd,\mm)\) be a given essentially non-branching \({\sf MCP}(K,N)\) space, where \(K\in\R\) and \(N\in(1,\infty)\).
Then \((\X,\sfd)\) is a length space, the measure \(\mm\) is uniformly locally doubling \cite{Stu:06} (thus, \((\X,\sfd)\) is proper)
and \((\X,\sfd,\mm)\) supports a weak local \((1,1)\)-Poincar\'{e} inequality \cite{vonRen:08} (cf.\ with \cite[Remark 2.6]{Cav:Mon:20}).
Given \(R>0\) and \(x\in\X\), we deduce from \cite[Corollary 4.17]{Cav:Mon:20} and Remark \ref{rmk:consist_Lapl} that \(\Delta\sfd_x^2\leq C(K,N,R)\mm\)
on \(B(x,R)\), where the constant \(0<C(K,N,R)<+\infty\) is given by
\[
C(K,N,R)\coloneqq\sup_{0<\theta<R}2\bigg(1+(N-1)\frac{r\,s'_{K/(N-1)}(r)}{s_{K/(N-1)}(r)}\bigg),
\]
the function \(s_{K/(N-1)}\) being defined as in \cite[Eq.\ (2-10)]{Cav:Mon:20}. Therefore, by applying Theorem \ref{thm:main_2}
we obtain that \(\Per(E\setminus B(x,r))\leq C(K,N,R)(2r)^{-1}\mm(E\cap B(x,r))+\Per(E)\) whenever \(E\subseteq\X\) is a set of finite perimeter,
\(x\in\X\) and \(r\in(0,R)\). One can easily deduce that \((\X,\sfd,\mm)\) has the deformation property stated in \cite[Definition 3.3]{Ant:Pas:Poz:Vio:23}.
The rest of the claim then follows from results of \cite{Ant:Pas:Poz:Vio:23}. Fix an isoperimetric set \(E\subseteq\X\). The bullet points
i) and ii) follow from \cite[Theorem 1.1]{Ant:Pas:Poz:Vio:23}, while iii) is a consequence of \cite[Theorem 3.9]{Ant:Pas:Poz:Vio:23}. Also,
iv) and v) are observed in \cite[Remark 3.10]{Ant:Pas:Poz:Vio:23}, while vi) follows from \cite[Corollary 1.3]{Ant:Pas:Poz:Vio:23}.
Therefore, the statement is achieved.
\end{proof}
%
%
%
%
%
%

%
%
%
%
\end{document}